\documentclass{article}
\usepackage{amsmath,amssymb,bm,amsthm}
\usepackage{mathrsfs}
\usepackage{enumerate}
\usepackage{hyperref}

\usepackage{graphicx}
\usepackage{mathtools}
\usepackage{xcolor}

\newcommand{\utot}{u^{\text{tot}}}
\newcommand{\uinc}{u^{\text{inc}}}
\newcommand{\usc}{u^+}
\newcommand{\diff}{\kappa}
\newcommand{\Hex}{H_{\Delta}^1(\Omega^+)}
\newcommand{\Hin}{H_{\mathcal{L}_\diff}^1(\Omega^-)}
\renewcommand{\div}{\operatorname{div}}
\newcommand{\R}{\mathbb{R}}

\newcommand{\sa}[1]{\left\{\!\!\left\{#1\right\}\!\!\right\}}
\newcommand{\sj}[1]{\left[\!\left[#1\right]\!\right]}
\newcommand{\tstep}{{\Delta t}}
\newcommand{\kr}{\lambda}
\newcommand{\Cinv}{C_{\text{inv}}}
\newcommand{\Mb}{\mathbf{M}}
\newcommand{\ub}{\mathbf{u}}
\newcommand{\Cb}{\mathbf{C}}
\newcommand{\Ib}{\mathbf{I}}

\renewcommand{\O}{\mathcal{O}}
\renewcommand{\Re}{\operatorname{Re}}

\newcommand{\Inner}[2]{\left(#1, #2 \right)_{\Omega}}
\newcommand{\inner}[2]{\left\langle #1,#2\right\rangle_{\Gamma}}

\newcommand{\mi}{\mathrm{i}}

\newcommand{\dbdf}{\delta_{\text{BDF2}}}
\newcommand{\dtr}{\delta_{\text{TR}}}
\newcommand{\dttr}{\delta_{\text{TTR}}}

\newcommand{\eps}{\mathrm{eps}}
\newtheorem{theorem}{Theorem}
\newtheorem{assumption}{Assumption}
\newtheorem{lemma}{Lemma}

\newcommand{\bl}[1]{ #1}

\begin{document}
\title{Implicit/explicit, BEM/FEM coupled scheme for acoustic
waves with the wave equation in the second order formulation}

\author{Lehel Banjai\footnote{ Maxwell Institute for Mathematical Sciences,  School of Mathematical \& Computer Sciences;  Heriot-Watt University,  Edinburgh EH14 4AS, UK}\\{\em In memory of a dear friend Francisco Javier Sayas}}


\maketitle
\begin{abstract}
Acoustic scattering of waves by bounded inhomogeneities in an unbounded homogeneous domain is considered. A symmetric coupled system  of time-domain boundary integral equations and the second order formulation of the wave equation is described. A fully discrete system consists of spatial discretization  by boundary and finite element methods (BEM/FEM), leapfrog time-stepping in the interior, and convolution quadrature for the boundary integral equations. Convolution quadrature is based on BDF2, trapezoidal rule, or a newly introduced truncated trapezoidal rule that has some favourable properties for both the implementation and quality of approximate solution. We give a stability and convergence analysis under a CFL conditon of the fully discrete system. The theoretical results are illustrated by numerical experiments in two dimensions.
\end{abstract}

\section{Introduction}
We consider the numerical simulation of the scattering of acoustic waves by a bounded inhomogeneity immersed in an infinite homogeneous domain. A time-domain boundary integral formulation (TBIE)  will be used in the unbounded homogeneous domain and will be coupled with the non-homogeneous wave equation (PDE) in the interior. In space we discretize the TBIE and PDE  by the boundary element and finite element methods respectively (BEM/FEM). The main aim and novelty of the paper is to present a new discretization using a second order formulation of the wave equation in the interior, explicit time-discretization of the PDE in the interior and implicit discretization of the BIE in the exterior and to prove stability and convergence of the scheme.  Previous approaches have either used the first order formulation \cite{AbJoRoTe:2011,BaLuSa:2015}, or a single implicit time-discretization for the interior and exterior equations \cite{FJS_FEMBEM,FaMo:2014,FaMo:2015}, or have not given a complete error analysis \cite{FaMo:2014,FaMo:2015,SOARES20071816}.   The reason for using different discretizations for the integral and PDE is the efficiency of the explicit time-discretization in the interior and the stability of the implicit scheme in the exterior. Note that the implicit time-discretization of the exterior problem is equivalent to convolution quadrature of the time-domain boundary integral operators; see \cite{Lub94}. We should say that the Johnson-N\'ed\'elec non-symmetric coupling \cite{JN} while cheaper to implement and showing  good performance in numerical experiments \cite{FaMo:2014,FaMo:2015}, has so far escaped the analysis in the time-domain.

There are many works on numerical methods for time-domain boundary integral equations and the wave equation; see the review \cite{Costabel:2017} for TDBIE and the recent book  \cite{Cohen_new} for discretizations of the wave equation. However, there are but a few works on the coupling in the time-domain; see previous paragraph. The earlier work \cite{BaLuSa:2015} is closest to the method developed here. While \cite{BaLuSa:2015} analysed only BDF2 based discretization for the exterior problem we allow other second order, A-stable linear multistep methods. Furthermore, we introduce a new time-discretization with many of the good properties of the trapezoidal scheme, but more easily implemented.  However, the main novelty is that we now use the second order formulation of the wave equation which allows for both simpler formulation and implementation. 

We next  give the statement of the problem and then proceed to introducing time-domain boundary integral operators and giving the weak formulation of the coupled system. In Section~4 we describe time discretization of TDBIE by convolution quadrature. Section 5 is central to the paper where a full stability and convergence analysis is given. In Section 6 we say a few words about the implementation and the choice of linear multistep method underlying the convolution quadrature. Finally we conclude with numerical experiments supporting the theory.  

\section{Statement of the problem}

Let $\Omega_j \subset \mathbb{R}^d$, $j = 1,\dots,J$, be open, bounded, connected Lipschitz domains such that their closures do not intersect and $d = 2,3$ is the spatial dimension. The inhomogeneity will be contained in $\Omega^- = \cup_{j = 1}^J \Omega_j$, $\Omega^+ = \mathbb{R}^d \setminus \overline{\Omega^-}$ is the exterior domain, and $\Gamma = \partial \Omega^-$ the boundary separating them; to simplify  the notation we will often write $\Omega$ for $\Omega^-$. The wave speed inside $\Omega^-$ can be variable $c(x)\colon \Omega^- \rightarrow \mathbb{R}$ with $c \in L^\infty (\Omega^-)$, $c > 0$ and $\|c\|_{L^\infty(\Omega^-)} < \bl{c_1}$ for some constant $c_1>0$. The diffusion coefficient is denoted by $\diff\colon \Omega^- \rightarrow \mathbb{R}_{\text{sym}}^{d \times d}$, where $\mathbb{R}_{\text{sym}}^{d \times d}$ is the space of symmetric real matrices and we assume that there exists a constant $\diff_0 > 0$ such that
\[
\xi^T \diff(x) \xi \geq \diff_0 |\xi|^2 \quad \text{for all } x \in \Omega^-, \;\xi \in \mathbb{R}^d.
\]
Furthermore, $\|\kappa(x)\|_2 \leq \kappa_1$ for some constant $\kappa_1> 0$ where we used the Euclidean matrix norm.

We also require the Sobolev spaces
\begin{equation}
  \label{eq:spaces}
  \begin{split}    
\Hin &= \{u \in H^1(\Omega^-) \,:\, \div (\diff \nabla u) \in L^2(\Omega^-)  \},\\
\Hex &= \{u \in H^1(\Omega^+) \,:\,  \Delta u \in L^2(\Omega^+)  \}.
  \end{split}
\end{equation}
The exterior unit normal is denoted by $\nu \in L^\infty(\Gamma)$,  the trace operators are denoted by $\gamma^{\pm} \colon H^1(\Omega^{\pm}) \rightarrow H^{1/2}(\Gamma)$ and the normal trace operators by $\partial^+_\nu\cdot = \nu . (\nabla\cdot) \colon \Hex \rightarrow H^{-1/2}(\Gamma)$ and  $\partial^-_{\diff,\nu}\cdot = \nu .(\diff \nabla  \cdot)\colon \Hin \rightarrow H^{-1/2}(\Gamma)$.

Further, we denote by $\|\cdot\|_{\Omega^{\pm}}$ the $L^2(\Omega^{\pm})$ norm. The $L^2$-sesquilinear products over $\Omega$ and $\Gamma$ are denoted by
\[
\Inner{u}{v}: = \int_\Omega u \overline{v} \qquad \inner{\varphi}{\psi} = \int_\Gamma \varphi \overline{\psi}.
\]
These products can be extended in the usual way to duality products $H^1(\Omega) \times (H^1(\Omega))'$ and $H^{1/2}(\Gamma) \times H^{-1/2}(\Gamma)$ or $H^{-1/2}(\Gamma) \times H^{1/2}(\Gamma)$. Finally the norms of the Hilbert spaces  $H^{-1/2}(\Gamma)$ and $H^{1/2}(\Gamma)$ are denoted by
\[
\|\cdot\|_{-1/2,\Gamma} := \|\cdot\|_{H^{-1/2}(\Gamma)} \qquad \|\cdot\|_{1/2,\Gamma} := \|\cdot\|_{H^{1/2}(\Gamma)}.
\]
Where the mapping properties of an operator are clear, we may use $\|\cdot\|$ to denote the natural norm.

Let $\uinc$ be the {\em incident wave} satisfying
\begin{equation}
  \label{eq:wave_inc}
  \partial_t^2 \uinc -\Delta \uinc = 0 \qquad \text{in }  \Omega^+\times \bl{[0,T]}.
\end{equation}
We assume that the supports \bl{of $\uinc(0)$ and $\partial_t \uinc(0)$} are contained in $\Omega^+$ and that the initial energy is finite, i.e.,  $\frac12\|\partial_t\uinc(0)\|^2_{L^2(\Omega^+)}+\frac12\|\nabla\uinc(0)\|^2_{L^2(\Omega^+)} < \infty$. The {\em total field} $\utot$ satisfies the homogeneous wave equation in the exterior domain
\begin{equation}
  \label{eq:utot_ex}
  \partial_t^2 \utot -\Delta \utot = 0 \qquad \text{in }  \Omega^+\times \bl{[0,T]}
\end{equation}
and the non-homogeneous wave equation in the interior
\begin{equation}
  \label{eq:utot_in}
\frac1{c^2}  \partial_t^2 \utot -\div (\diff \nabla \utot) = f \qquad \text{in }  \Omega^-\times \bl{[0,T]},
\end{equation}
where the support of $f(t) \in L^2(\Omega^-)$ is contained in $\Omega^-$ for all $t \in [0,T]$. The system is completed by  transmission conditions on $\Gamma$
\begin{equation}
  \label{eq:utot_transm}
\gamma^+\utot = \gamma^- \utot, \qquad  \partial_\nu^+ \utot =   \partial^-_{\diff,\nu} \utot 
\end{equation}
and initial conditions
\begin{equation}
  \label{eq:utot_init_plus}  
\utot(0) = \uinc(0), \; \partial_t\utot(0) = \partial_t\uinc(0)
\qquad \text{ in } \Omega^+
\end{equation}
and 
\begin{equation}
  \label{eq:utot_init_minus}  
\utot(0) = u_0, \; \partial_t \utot(0) = v_0\qquad 
\text{ in } \Omega^-,
\end{equation}
for given initial data $u_0$ and $v_0$ with supports contained in $\Omega^-$.
Note that due to the assumptions on $\uinc$, $u_0$, and $v_0$,  $\utot(0)$ \bl{and $\partial_t\utot(0)$ are zero in a neighbourhood of $\Gamma$}.

In the exterior, as is common, instead of $\utot$ we will be computing the {\em scattered field} 
\begin{equation}
  \label{eq:scattered}
\usc = \utot-\uinc \qquad \text{in }\Omega^+.
\end{equation}
Furthermore, to simplify notation, we denote the total field in the interior by
\begin{equation}
  \label{eq:tot_interior}
u = \utot \qquad \text{in }\Omega^-.
\end{equation}

\begin{figure}
  \centering
  \includegraphics[width=.35\textwidth]{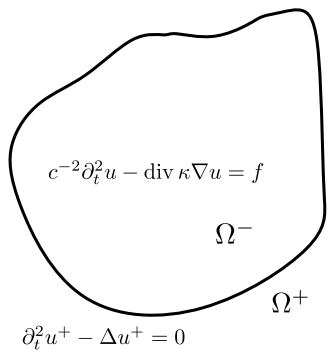}
  \caption{The geometrical setting of the coupling problem. }
  \label{fig:coupling}
\end{figure}

Putting everything together we are solving the following problem: Find $u(t) \in \Hin$, $\usc(t) \in \Hex$ such that for  $t\in [0,T]$
\begin{subequations}\label{eq:prob}
  \begin{equation}
    \label{eq:interior}
    c^{-2}\partial_t^2 u-\div (\diff \nabla u) = f \qquad \text{in } \Omega^-,
  \end{equation}
  \begin{equation}
    \label{eq:exterior}
    \partial_t^2 \usc-\Delta \usc = 0 \qquad \text{in } \Omega^+,
  \end{equation}
  \begin{equation}
    \label{eq:trace_cond}
    \gamma^- u= \gamma^+ \usc + \beta_0 \qquad \text{on } \Gamma,
  \end{equation}
  \begin{equation}
    \label{eq:trace_cond1}
    \partial_{\diff,\nu}^- u= \partial_{\nu}^- \usc + \beta_1 \qquad \text{on } \Gamma,
  \end{equation}
\end{subequations}
where $\beta_0 = \gamma^+ \uinc$ and $\beta_1 = \partial_\nu \uinc$ are the traces of the incident wave onto $\Gamma$; \bl{see Figure~\ref{fig:coupling}}. As at time $t = 0$, $\uinc$ has not reached the domain $\Omega^-$, $\usc$  satisfies the zero initial condition
\[
  \usc(0) = \partial_t \usc(0) = 0 \qquad \text{ in } \Omega^+.
\]
and $u$ the inital condition
\[
  u(0) = u_0,\; \partial_t u(0) = v_0 \qquad \text{in } \Omega^-.
\]
\bl{The smoothness requirements in time of the data $\uinc$ and $f$ will be given in the next section once we describe the boundary-field formulation of the problem. }

\section{One sided time-convolutions and time-domain boundary integral operators}
We intend to represent the solution in the unbounded domain by time-domain boundary integral potentials. To introduce some notation and properties of these we first consider the one-sided convolution
\[
K(\partial_t) g(t) = \int_0^t k(t-\tau) g(\tau)\bl{d\tau},
\]
where the kernel $k$ is given via its Laplace transform 
\[
K(s) = \mathscr{L}\{ k\}(s) = \int_0^\infty e^{-st}K(s) ds \qquad \Re s >0.
\]
If the kernel $k$ and the data $g$ are integrable, the meaning of the \bl{convolution} is clear. However often this is not the case and instead we only know that $K$ is analytic for $\Re s> 0$ and satisfies the bound 
\begin{equation}
  \label{eq:Kbound}
|K(s)| \leq C(\sigma) |s|^\mu ,\qquad \forall \Re s \geq \sigma > 0.  
\end{equation}
In this case, following \cite{Lub94}, the convolution is defined by 
\[
K(\partial_t)g(t) := \mathscr{L}^{-1}\{ K G\}(t) = \frac1{2\pi\mi}\int_{\sigma+\mi \R} e^{st} K(s) G(s) ds,
\]
where $G$ is the Laplace transform of data $g$, $\mathscr{L}^{-1}$ is the inverse Laplace transform, and the operational notation $K(\partial_t)$ emphasizes the importance of the kernel in the Laplace domain. Note that if $g \in C^m(\bl{[0,\infty)})$ \bl{and its derivatives upto order $m$} are polynomially bounded  with $g(0) = g'(0) = \dots = g^{(m-1)}(0) = 0$, then
\[
\left|\mathscr{L} g (s) \right| \leq |s|^{-m} \int_0^\infty e^{-\sigma t} g^{(m)}(t) dt.
\]
If $m > \mu+1$, the inverse Laplace transform and the Cauchy integral formula imply that  $K(\partial_t) g \in C(\bl{[0,\infty)})$ and $K(\partial_t) g(0) = 0$. 
If $g$ is only defined on a finite interval $[0,T]$, we can extend it by the Taylor polynomial $\sum_{j = 0}^m \frac1{j!}g^{(j)}(T)(t-T)^m$ to $t > T$ and again define $K(\partial_t)g$ via the Laplace domain as above.

For $K(s) = s$ and sufficiently smooth $g$, $K(\partial_t) g = \partial_t g$ is indeed the time-derivative justifying the operational notation. Furthermore, for sufficiently smooth $g$ and operators $K_1$ and $K_2$ satisfying bound of the form \eqref{eq:Kbound}, the composition property $K_2K_1(\partial_t) g = K_2(\partial_t) K_1(\partial_t) g$ holds.

The theory of time domain boundary integral operators seemlesly fits into this framework; see \cite{Lub94}. Namely, the Laplace domain single layer and double layer potentials are given by
\[
S(s) \varphi (x) :=  \int_\Gamma \mathcal{K}(|x-y|,s) \varphi(y) d\Gamma_y  \qquad x \in \R^d \setminus \Gamma
\]
and
\[
D(s) \varphi (x) :=  \int_\Gamma\left[\partial_{\nu_y} \mathcal{K}(|x-y|,s)\right] \varphi(y)d\Gamma_y \qquad x \in \R^d \setminus \Gamma,
\]
with the kernel $\mathcal{K}(r,s)$  given by
\[
\mathcal{K}(r,s) :=  \left\{
  \begin{array}{cc}
    \displaystyle \frac1{2\pi}K_0(sr)& d = 2\vspace{.1cm}\\
    \displaystyle\frac{e^{-sr}}{4\pi r}& d = 3.
  \end{array}
\right.,
\]
\bl{where $K_0(\cdot)$ is a modified Bessel function \cite[Chapter 10]{NIST}.}
The single and double-layer potentials have the mapping properties
\[
S(s) \colon H^{-1/2}(\Gamma) \to H^1(\R^d \setminus \Gamma), \qquad
D(s) \colon H^{1/2}(\Gamma) \to H^1(\R^d \setminus \Gamma)
\]
and satisfy bounds
\[
\|S(s)\| \leq C(\sigma) \frac{|s|}{\Re s}, \quad
\|D(s)\| \leq C(\sigma) \frac{|s|^{3/2}}{\Re s},\qquad \Re s\geq \sigma >0.
\]
For these and other estimates in the Laplace domain see the original papers of Bamberger and Ha Duong  \cite{BaHa:1986a,BaHa:1986b}, crucial progress made in \cite{LaSa:2009a} and the books \cite{lbbook,Sayas:2016}.

Hence, as described in the discussion on one-sided convolutions, the time-domain single layer and double layer time domain boundary integral potentials $S(\partial_t)$ and $D(\partial_t)$ are well-defined.  For sufficiently smooth data $\varphi$, $\psi$ they are given by the explicit formula
\[
S(\partial_t) \varphi (x,t) := \int_0^t \int_\Gamma k(|x-y|,t-\tau) \varphi(y,\tau)d\Gamma_y d \tau \qquad x \in \R^d \setminus \Gamma
\]
and
\[
D(\partial_t) \psi (x,t) := \int_0^t \int_\Gamma\left[\partial_{\nu_y} k(|x-y|,t-\tau)\right] \psi(y,\tau)d\Gamma_y d \tau \qquad x \in \R^d \setminus \Gamma,
\]
where the kernel $k$ is given by
\[
k(r,t) :=  \left\{
  \begin{array}{cc}
    \displaystyle\frac{H(t-r)}{2\pi\sqrt{t^2-r^2}}& d = 2\vspace{.1cm}\\
    \displaystyle\frac{\delta(t-r)}{4\pi r}& d = 3,
  \end{array}
\right.
\]
and $\delta(\cdot)$ and $H(\cdot)$ are the Dirac delta and Heaviside distributions respectively. Note that the kernel of the Laplace domain operators $\mathcal{K}$ is the Laplace transform of $k$.

If $U$ solves the Laplace transformed wave equation 
\[
s^2 U -\Delta  U = 0 \qquad \text{in } \Omega^+,
\]
then Kirchhoff's representation formula gives $u$ in terms of boundary integral potentials
\[
U = -S(s)\partial_\nu^+ U +D(s) \gamma^+ U. 
\]
\bl{The} corresponding formula holds in the time-domain, and we represent $\usc$ as a combination of boundary integral potentials using the Kirchhoff formula
\begin{equation}
  \label{eq:u_repr_FB}  
\usc = S(\partial_t)\varphi + \partial_t^{-1}D(\partial_t) \psi,
\end{equation}
where  the densities are exterior traces of the scattered field
\begin{equation}
  \label{eq:traces_FB}
 \varphi = -\partial^+_\nu \usc, \qquad \psi = \gamma^+\partial_t \usc.
\end{equation}
\bl{The densities $\varphi$ and $\psi$ will be the unkowns in addition to the interior solution $u$. } The reason for using this particular form of the densities $\varphi$ and $\psi$ will will only become apparent  once we obtain a convenient boundary-field formulation of the transmission problem.

\bl{Next, by taking appropriate traces of the boundary potentials we obtain the boundary integral operators.} First, we define the boundary average
\[
\sa{\cdot} = \tfrac12(\gamma^+\cdot+\gamma^-\cdot)
\]
and the boundary jump
\[
\sj{\cdot} = \gamma^+\cdot-\gamma^-\cdot
\]
and recall the jump properties of the layer potentials 
\begin{equation}
  \label{eq:jump_prop}
  \begin{split}
    \sj{S(\partial_t)\varphi} &= 0,\\
    \sj{\partial_\nu S(\partial_t)\varphi} &= -\varphi,\\
    \sj{D(\partial_t)\psi} &= \psi,\\
    \sj{\partial_\nu D(\partial_t)\psi} &= 0.
  \end{split}
\end{equation}
Finally we define four boundary integral operators 
\[
\begin{split}
V(\partial_t)\varphi &= \sa{S(\partial_t)\varphi} = \gamma^{\pm}S(\partial_t)\varphi\\
K(\partial_t)\psi &= \sa{D(\partial_t)\psi}\\
K^t(\partial_t)\varphi &= \sa{\partial_\nu S(\partial_t)\varphi}\\
W(\partial_t)\psi &= -\sa{\partial_\nu D(\partial_t)\psi} = -\partial^{\pm}_\nu D(\partial_t)
\end{split}
\]
where $V$, $K$, $K^t$, $W$ are the single layer, double layer, transposed double layer and hypersingular boundary integral operators. Note that
\[
\begin{split}
  V(s) \colon& H^{-1/2}(\Gamma) \rightarrow H^{1/2}(\Gamma)\\
  K(s) \colon& H^{1/2}(\Gamma) \rightarrow H^{1/2}(\Gamma)\\
  K^t(s) \colon& H^{-1/2}(\Gamma) \rightarrow H^{-1/2}(\Gamma)\\
 W(s) \colon& H^{1/2}(\Gamma) \rightarrow H^{-1/2}(\Gamma)
\end{split}
\]
and that each of the operators satisfies a bound of the form \eqref{eq:Kbound} with $\mu = 1, 3/2, 3/2, 2$ respectively; see \cite{Sayas:2016}.

The above definitions and jump properties imply
\[
\gamma^+ D(\partial_t) \psi = \frac12 \psi + K(\partial_t) \psi
\]
and
\[
\partial_\nu^+ S(\partial_t) \varphi = -\frac12 \varphi + K^t(\partial_t) \varphi.
\]
In particular, differentiating  \eqref{eq:u_repr_FB} in time, taking the trace 
gives
\[
\gamma^+\partial_t \usc = \partial_t V(\partial_t)\varphi +  \frac12 \psi+K(\partial_t) \psi  
\]
\bl{and applying the transmission condition \eqref{eq:trace_cond1}
\begin{equation}
  \label{eq:bie1}
-\partial_t \beta_0=  -\gamma^-\partial_t \usc +\partial_t V(\partial_t)\varphi +  \frac12 \psi+K(\partial_t) \psi.  
\end{equation}}
Further, taking the normal trace of \eqref{eq:u_repr_FB}
gives
\begin{equation}
  \label{eq:bie2}
- \frac12 \varphi -K^t(\partial_t)\varphi + \partial_t^{-1}W(\partial_t) \psi = 0.  
\end{equation}

\bl{Finally, testing the strong formulation \eqref{eq:interior} in the interior by $v \in H^1(\Omega^-)$  and using the transmission condition \eqref{eq:trace_cond1} we obtain the weak formulation in the interior
\[
\Inner{c^{-2}\partial_t^2 u}{ v}+\Inner{\kappa \nabla u}{\nabla v}
+\inner{\varphi}{\gamma^- v} = \Inner{f}{v}+\inner{\beta_1}{\gamma^- v}.
\]
}
\bl{Thus, combining this with the boundary integral equations \eqref{eq:bie1} and \eqref{eq:bie2}}, the boundary/field formulation of the scattering problem reads: Find $u(t) \in \Hin$, $\varphi \in H^{-1/2}(\Gamma)$, $\psi \in H^{1/2}(\Gamma)$ such that for almost all $t \in [0,T]$
\begin{subequations}\label{eq:prob}
  \begin{equation}
    \label{eq:pde_weak}
\Inner{c^{-2}\partial_t^2 u}{ v}+\Inner{\kappa \nabla u}{\nabla v}
+\inner{\varphi}{\gamma^- v} = \Inner{f}{v}+\inner{\beta_1}{\gamma^- v} 
  \end{equation}
  \begin{equation}
    \label{eq:bie_weak}
    \begin{pmatrix}
 -\gamma^-\partial_t u\\
0     
    \end{pmatrix}
+
\begin{pmatrix}
 0& \frac12  I \\
-\frac12 I &0
\end{pmatrix}
\begin{pmatrix}
  \varphi \\\psi
\end{pmatrix}
 + B(\partial_t) \begin{pmatrix}
  \varphi \\\psi
\end{pmatrix}
 =
 \begin{pmatrix}
-\partial_t \beta_0\\0
 \end{pmatrix}
  \end{equation}
\end{subequations}
for all $v \in H^1(\Omega^-)$, where the second equality is understood in $H^{1/2}(\Gamma) \times H^{-1/2}(\Gamma)$, the initial data are given by
\[
u(0) = u_0 \;(\text{in } H^1(\Omega^-)) \qquad \partial_t u(0) = v_0  \;(\text{in } L^2(\Omega^-)), 
\]
and $B$ is the Calder\'on operator 
\begin{equation}
  \label{eq:calderon}
B(\bl{\partial_t}) = 
\begin{pmatrix}
  \bl{\partial_t} V(\bl{\partial_t}) & K(\bl{\partial_t})\\
-K^{t}(\bl{\partial_t}) & \bl{\partial_t}^{-1}W(\bl{\partial_t})
\end{pmatrix}.
\end{equation}
The bounds on the constituent operators in the definition of $B$ imply that the operator $B(s)\colon H^{-1/2}(\Gamma) \times H^{1/2}(\Gamma) \to H^{1/2}(\Gamma) \times H^{-1/2}(\Gamma)$ itself satisfies a bound of the form \eqref{eq:Kbound}:
\begin{equation}
  \label{eq:Bcont}
  \|B(s)\| \leq C(\sigma) |s|^2 \qquad \Re s\geq \sigma > 0.
\end{equation}

Next, we state a crucial property of the Calder\'on operator in the frequency domain; for a proof see \cite{BaLuSa:2015}. 

\begin{lemma}\label{lem:calderon_pos}
There exists $\beta > 0$ so that the Calder\'on operator \eqref{eq:calderon} satisfies
\[
\Re \left \langle
  \begin{pmatrix}
      \varphi \\ \psi
  \end{pmatrix}, B(s) 
  \begin{pmatrix}
      \varphi \\ \psi
  \end{pmatrix}
\right \rangle_\Gamma
 \geq \beta\, \min(1,|s|^2) \frac{\Re s}{|s|^2} \left(\|\varphi\|^2_{-1/2,\Gamma} + \|\psi\|^2_{1/2,\Gamma}\right)
\]
for  $\Re s > 0$ and for all $\varphi \in H^{-1/2}(\Gamma)$ and $\psi \in H^{1/2}(\Gamma)$.
\end{lemma}

\bl{Let us briefly consider the existence and uniqueness of the solution to the coupled system \eqref{eq:prob}.  Following \cite{Sayas:2016}, for a Banach space $X$, let us denote by $TD(X)$ the space of causal, tempered distributions with polynomially bounded Laplace transform. Namely, for $f \in TD(X)$, there exists a non-increasing function $C_F(x) > 0$ with $C_F(x) \leq C x^{-m}$ for some $m$ and $x \in (0,1)$, such that 
 \[
 \|\mathcal{L}\{f\}(s)\|_X \leq C(\sigma) |s|^m, \qquad \Re s \geq \sigma > 0.
 \]
For vanishing $f$, $u_0$ and $v_0$, the existence and uniqueness of the solution of \eqref{eq:prob}  can be shown  by means of the Laplace transform and the above coercivity result under very weak conditions on the smoothness of the data; for details see \cite[Chapter 7]{lbbook} and also \cite[Proposition~2.1]{FJS_FEMBEM}. Namely, if $\beta_0 \in TD(H^{1/2}(\Gamma))$, $\beta_1 \in TD(H^{-1/2}(\Gamma))$, there exists a unique solution $u \in TD(\Hin)$, $\varphi \in TD(H^{-1/2}(\Gamma))$, $\psi \in TD(H^{1/2}(\Gamma))$ of \eqref{eq:prob}.

For vanishing $\beta_0$ and $\beta_1$, we first let $\tilde \Omega_T$ be a bounded, Lipschitz domain such that $\Omega^- \subset \tilde \Omega_T$ and
\[
\operatorname{dist}(\partial\Omega^-,\partial\tilde \Omega_T) > T.
\]
The existence of the unique weak solution $\tilde u(t) \in H^1_0(\tilde \Omega_T)$ of
\[
\partial_t^2 \tilde u + L \tilde u = f\qquad   
\]
where
\[
L \tilde u = \left\{
  \begin{array}{cc}
-\div \kappa \nabla \tilde u & \text{in } \Omega^-\\
-\Delta \tilde u & \text{in } \Omega^+
  \end{array}
\right.
\]
for $f \in L^2(0,T; L^2(\tilde \Omega_T))$ with initial data $\tilde u(0) = u_0 \in H^1_0(\tilde\Omega_T)$, $\partial_t \tilde u(0) = v_0 \in L^2(\tilde\Omega_T)$ (with $u_0$ and $v_0$ extended by $0$ to $\tilde \Omega_T$) follows by classical means, see, e.g., \cite{Evans}. Note that due to the finite propagation of waves and the choice of $\tilde \Omega_T$, $\tilde u(t)  \equiv 0$ in a vicinity of $\partial \tilde \Omega_T$ for $t < T$. Setting  $u = \tilde u|_{\Omega^-}$, $\varphi = -\partial^+_\nu u$, and $\psi = \gamma^+\partial_t u$ gives a solution of \eqref{eq:prob}. Uniqueness again following by the Laplace transform as in \cite{FJS_FEMBEM,lbbook}. 

The above shows existence and uniquenness of the solution to the coupled problem \eqref{eq:prob}. Much stronger requirements on the smoothness of both data and solution will be made in later section to analyse the convergence of numerical methods. }

\section{Time-discretization of time-domain boundary integral operators}

Convolution quadrature (CQ) is a time-discretization of convolutions $K(\partial_t)g$ based on an A-stable linear multistep method for symbols $K$ satisfying a bound \eqref{eq:Kbound}; see \cite{Lub94}.   Just like the one-sided convolution, it is defined via the Laplace domain. Namely, given a fixed time-step $\tstep > 0$ we define
\begin{equation}
  \label{eq:CQ_def}  
K(\partial_t^{\tstep}) g(t) :=  \frac1{2\pi\mi}\int_{\sigma+\mi \R} e^{st} K(s^{\tstep}) G(s) ds,
\end{equation}
where
\begin{equation}
  \label{eq:sdt}  
s^{\tstep} = \frac{\delta(e^{-s \tstep})}{\tstep}
\end{equation}
and $\delta(\zeta)$ is a generating function of an A-stable linear multistep method. Namely, we assume that   \bl{the method is of order $p \geq 1$}
\begin{equation}
  \label{eq:delta_order}  
\delta(e^{-z}) = z+\O(z^{p+1})
\end{equation}
and \bl{is A-stable}
\begin{equation}
  \label{eq:Astab}  
\Re \delta(\zeta) > 0 \qquad \text{for } |\zeta| < 1.
\end{equation}
As the Dahlquist's second barrier tells us that $p \leq 2$, we will mainly be concerned with the following second order methods: the second order backward difference formula (BDF2)
\[
\dbdf(\zeta) = (1-\zeta)+\frac12(1-\zeta)^2
\]
and the trapezoidal rule
\[
\dtr(\zeta) = 2\frac{1-\zeta}{1+\zeta} = \sum_{j = 0}^\infty 2^{-j}(1-\zeta)^{j+1}.
\]
As we will look at another class of linear multistep methods, we state the assumptions they need to satisfy.

\begin{assumption}\label{ass:delta}
  Let $\delta(\zeta)$ satisfy \eqref{eq:delta_order} for $p \geq 1$ and \eqref{eq:Astab}. Furthermore, let $\delta(\zeta)$  be either analytic for $|\zeta| \leq 1$ or \bl{be} the generating function of the trapezoidal rule.
\end{assumption}

For a sufficiently smooth $g$, convolution quadrature inherits the approximation property of the underlying linear mulstistep method. To present the required smoothness we need the space
\begin{equation}
  \label{eq:Wspace}
  \begin{split}    
W_0^m(\R) \coloneqq \{&g \in C^{m-1}(\R)\; :\; g \equiv 0 \text{ in } (-\infty,0), \\&g \text{ polynomially bounded }, g^{(m)} \in L^1_{\text{loc}}(\R)\}.
  \end{split}
\end{equation}
For $g \in W_0^m(\R)$ we then have 
\[
K(\partial_t)g-K(\partial_t^{\tstep})g = \O(\tstep^p),
\]
with $m > \max(2\mu+3,\mu+4)$ for the trapezoidal rule and $m > \mu+p+2$ for other methods satisfying Assumption~\ref{ass:delta}. For the proof of this result see \cite{Lub94} (not including the trapezoidal rule for $\mu \geq 0$), \cite{FJStrap} (for the trapezoidal rule), and \cite{lbbook} (for all the cases). 

The above brief introduction to CQ does not indicate how to implement the method; for details and short codes see \cite{lbbook}. For our purposes here,  let us just note that $K(\partial_t^{\tstep})g$ is given by a discrete convolution
\[
K(\partial_t^{\tstep})g(t_n) = \sum_{j = 0}^n \omega_{n-j}(K) g(t_j).
\]
The convolution weights $\omega_j(K)$ can be expressed  by the contour integral
\[
\omega_j(K) = \frac1{2\pi \mi} \oint_{\mathcal{C}} K(\delta(\zeta)/\tstep) \zeta^{-j-1} d\zeta,
\]
with $\mathcal{C}$  a circle of radius $0 < \lambda < 1$. Discretizing the contour integral by the compound trapezoidal rule gives the approximation
\begin{equation}
  \label{eq:omega_quad}
 \omega_j(K) 
 \approx \frac{\kr^{-j}}{N+1} \sum_{\ell = 0}^N 
K\left(\frac{\delta(\kr\zeta_{N+1}^{-\ell})}{\Delta  t}\right)
\zeta_{N+1}^{\ell j},
\end{equation}
where  $\zeta_{N+1} = e^{\frac{2\pi \mi}{N+1}}$ and $0 < \kr <1$.  The error commited is $\O(\lambda^{N+1})$ and all $N$ weights can be computed in $\O(N \log N)$ time using FFT.  For numerical stability reasons $\lambda$ is chosen greater than $\text{eps}^{\frac1{2N+1}}$ where $\text{eps}$ is the machine precision; see \cite{Lub94,lbbook}.

A crucial property of CQ is that it inherits the positivity property of the kind satisfied by the Calder\'on operator; for a proof see \cite{BaLuSa:2015}.

\begin{theorem}\label{thm:herglotz}
Let $K(s)$ be an analytic family for $\Re  s > 0$ of linear operators between a Hilbert space $X$ and its dual $X'$ satisfying the bound
\[
\|K(s)\| \leq C_0(\sigma) |s|^\mu \qquad \Re s \geq \sigma > 0
\]
for some $\mu \in \R$ and any $\sigma > 0$. Denoting by $\langle \cdot, \cdot \rangle$ the duality product, let 
for any $\sigma > 0$,
\begin{equation}
  \label{eq:Ks_pos}
  \langle \varphi, K(s)\varphi \rangle \geq C_1(\sigma)\|s^\eta\varphi\|_X^2 \qquad \forall \varphi \in X,
\end{equation}
 $\Re s \geq \sigma > 0$, and $C_1(x)$ a non-increasing function of $x$,   and some $\eta \in \R$. Considering CQ based on $\delta(\zeta)$ satisfying Assumption~\ref{ass:delta},  $\sigma \tstep > 0$ small enough, $\varrho = e^{-\sigma \tstep}$ and any finite series $\varphi_j \in X$ the following holds 
\[
\sum_{j = 0}^\infty \varrho^{2j}\left \langle \varphi_j, K(\partial_t^{\Delta t}) \varphi(t_j)\right\rangle \geq C_2(\sigma) \sum_{j = 0}^\infty \varrho^{2j} \|(\partial_t^{\Delta t})^\eta \varphi_j\|^2_X,
\]
with some positive contant $C_2(\sigma)$ depending on $\sigma$ and the choice of $\delta(\zeta)$. 
\end{theorem}

A related result we will need is stated in the next lemma. It has been used in the proofs of \cite{BaLuSa:2015}, see also \cite{lbbook}.

\begin{lemma}\label{lem:parseval}
  Under the conditions of the previous theorem, given finite sequences $\varphi_j \in X$ and $\psi_j \in X'$, and $\sigma > 0$, for any $\varrho \in (0,1)$ the following holds
\[
\sum_{j = 0}^\infty \varrho^{2j} |\langle \varphi_j,\psi_j\rangle| \leq
\frac12\sum_{j = 0}^{\infty} \varrho^{2j}\left( \|(\partial_t^{\tstep})^{-1} \varphi(t_j)\|_{X'}^2+\|\partial_t^{\tstep} \psi(t_j)\|_X^2\right).
\]
\end{lemma}

Both results are proved using Parseval's formula.


\section{\bl{Fully discrete system and convergence analysis}}

To present the full discretization it will be useful to define the second order central difference approximations of the first and second time derivative:
\[
[Du]_n :=  \frac1{2\tstep}(u_{n+1}-u_{n-1}), \qquad
[D^2u]_n :=  \frac1{\tstep^2}(u_{n+1}-2u_n+u_{n-1}).
\]
For the discretization in space we let $\bl{X_h} \subset H^1(\Omega)$, $\bl{X_h^{-\frac12}}  \subset H^{-1/2}(\Gamma)$, $\bl{X_h^{\frac12}} \subset H^{1/2}(\Gamma)$ be families of subspaces indexed by the spatial meshwidth $h > 0$. We assume that $\bl{X_h}$ satisfies the inverse inequality
\begin{equation}
  \label{eq:inverse_ineq}
  \|\kappa^{1/2}\nabla u\|_{\Omega} \leq \Cinv h^{-1}\|u\|_\Omega \qquad \text{for all } u \in X_h
\end{equation}
for some $\Cinv > 0$. 

The Galerkin discretization of the Calder\'on operator  denoted by $B_h(s) \colon X_h^{-\frac12} \times X_h^{\frac12} \to (X_h^{-\frac12})' \times (X_h^{\frac12})'$ is defined by
\[
\inner{
  \begin{pmatrix}
    z\\w
  \end{pmatrix}}
{B_h(s)
  \begin{pmatrix}
    \eta \\ \mu
  \end{pmatrix}
} = 
\inner{
  \begin{pmatrix}
    z\\w
  \end{pmatrix}}
{B(s)
  \begin{pmatrix}
    \eta \\ \mu
  \end{pmatrix}
} \qquad \forall z \in X_h^{-\frac12}, w \in X_h^{\frac12}.
\]
Furthermore, we denote by 
\[
\Pi_h^2 \colon H^{1/2}(\Gamma) \to (X_h^{-\frac12})' \quad
\Pi_h^3 \colon H^{-1/2}(\Gamma) \to (X_h^{\frac12})'
\]
 the orthogonal projectors.

The fully discrete system then reads: Find $u_{n+1}^h \in X_h$, $\varphi^h_n \in X_h^{-\frac12}$, $\psi_n^h \in X_h^{\frac12}$ for $n = 1,\dots,N-1$ such that
\begin{subequations} \label{eq:full_disc} 
  \begin{equation}
    \label{eq:pde_disc}
    \begin{split}      
\Inner{c^{-2}[D^2 u^h]_n}{ v}+\Inner{\kappa \nabla u^h_n}{\nabla v}
+&\inner{\varphi^h_n}{\gamma^- v}\\ =& \Inner{f(t_n)}{v}+\inner{\beta_1(t_n)}{\gamma^- v}
    \end{split}
  \end{equation}
  \begin{equation}
    \label{eq:bie_disc}
    \begin{split}     
\inner{ -\gamma^-[D u^h]_n}{z}&
 +\tfrac12 \inner{\psi^h_n}{z}-\tfrac12\inner{\varphi^h_n }{w}
 \\&+ \inner{B(\partial^\tstep_t) \begin{pmatrix}
  \varphi^h \\\psi^h
\end{pmatrix}(t_n)}{
\begin{pmatrix}
  z\\w
\end{pmatrix}}
=
-\inner{\partial_t \Pi_h^2\beta_0(t_n)}{z}
\end{split}
  \end{equation}
\end{subequations}
for all $v \in X_h$, $z \in X_h^{-\frac12}$, $w \in X_h^{\frac12}$. To complete the system we need to define the initial data. We set $u_0^h = R_h u_0$ and using the equation in the interior we set 
\[
  \begin{split}    
u_1^h &= u_0^h+\tstep R_h v_0+\frac12 \tstep^2 R_h \partial_t^2 u(0)\\
 &= u_0^h+\tstep R_h v_0+\frac12 \tstep^2 R_h (c^2\div (\kappa \nabla u_0)-f(0)).
  \end{split}
\]
Here $R_h\colon \Hin \to X_h$ is the elliptic projector satisfying 
\begin{equation}
  \label{eq:Rh}  
\Inner{\kappa \nabla R_h u}{\nabla v}+\Inner{R_hu}{v}  = \Inner{\kappa \nabla u}{\nabla v}+\Inner{u}{v} \qquad \forall v \in X^1_h.
\end{equation}
We also set $\varphi_0^h = 0$ and $\psi_0^h = 0$ since at time $t = 0$, the wave has not yet reached the boundary $\Gamma$. 

We first prove the stability of the above system under perturbations. Stability will be shown using the discrete energy in $\Omega^-$
\[
E_n(u^h) = \frac12 \left \|\frac{u^h_n-u^h_{n-1}}{c\tstep} \right\|^2_{\Omega}
+ \frac12 \Inner{\kappa \nabla u^h_n}{\nabla u^h_{n-1}}.
\]
Standard calculation shows that the inverse inequality \eqref{eq:inverse_ineq} implies
\[
E_n(u^h) \geq \frac12\left(1-\frac{1}{4}\Cinv^2 c_1^2\tstep^2 h^{-2}\right)\left\| \frac{u^h_n - u^h_{n-1}}{c\tstep} \right\|^2_{\Omega}+ \frac{1}{2}\left\|\kappa^{1/2} \frac{\nabla u^h_n + \nabla u^h_{n-1}}{2} \right\|^2_{\Omega}.
\]
Therefore under the CFL condition 
\begin{equation}
  \label{eq:CFL}
  \tstep < \frac{\sqrt{2} h}{\Cinv c_1 }
\end{equation}
the discrete energy is positive and
\[
E_n(u^h) \geq \frac14\left\| \frac{u^h_n - u^h_{n-1}}{c\tstep} \right\|_{\Omega}^2+ \frac{1}{2}\left\|\kappa^{1/2} \frac{\nabla u^h_n + \nabla u^h_{n-1}}{2} \right\|^2_{\Omega}.
\]

\begin{theorem}\label{thm:stab}
  Let  $g_n \in L^2(\Omega)$, $\rho_n \in H^{1/2}(\Gamma)$, $\sigma_n \in H^{-1/2}(\Gamma)$, $n = 1\dots$,  and $u_0^h, u_1^h \in X_h$ be given.   Under the CFL condition \eqref{eq:CFL}, $\tstep \leq \tstep_0$ for some fixed $\tstep_0> 0$  and with $\varphi_0^h = \psi_0^h = 0 $,  the system, $n = 1,\dots,N$,
\[
    \begin{split}      
\Inner{c^{-2}[D^2 u^h]_n}{ v}+\Inner{\kappa \nabla u^h_n}{\nabla v}
+\inner{\varphi^h_n}{\gamma^- v} &= \Inner{g_n}{v}\\
    \begin{pmatrix}
 -\gamma^-[D u^h]_n\\
0     
    \end{pmatrix}
+
\begin{pmatrix}
\frac12\psi^h_n\\  -\frac12\varphi^h_n 
\end{pmatrix}
 + B_h(\partial^\tstep_t) \begin{pmatrix}
  \varphi^h \\\psi^h
\end{pmatrix}(t_n)
 &=
 \begin{pmatrix}
\Pi_h^2\rho_n\\\Pi_h^3\sigma_n
 \end{pmatrix}
\end{split}
\]
 for all $v \in X_h$ has a unique solution. Further, the following stability bound holds
\[
\begin{split}    
  E_{N+1}+\tstep \sum_{n = 0}^{N}
&\left(\|(\partial_t^{\tstep})^{-1} \varphi^h(t_n)\|_{-1/2,\Gamma}^2+ \|(\partial_t^{\tstep})^{-1} \psi^h(t_n)\|_{1/2,\Gamma}^2\right)\\ &\leq C(T)(E_1+R_N),
  \end{split}
\]
where
\[
  \begin{split}    
    R_N =&  \frac{1}{2\beta}\tstep\sum_{n = 0}^{\infty} \varrho^{2n}\left(
\|\partial_t^{\tstep}\rho(t_n)\|_{1/2,\Gamma}^2+
\|\partial_t^{\tstep}\sigma(t_n)\|_{-1/2,\Gamma}^2
\right)
\\ &+ \tstep\frac12c_1^2\sum_{n = 0}^N \varrho^{2n}\|g_n\|_{\Omega}^2,
  \end{split}
\]
and $C(T) > 0$ is a constant depending on the final time $T = N\tstep$. 
\end{theorem}
\begin{proof}
As the linear system to be solved is square, the stability bound implies uniqueness and hence existence. To obtain the stability bound we test the system with $v = [Du^h]_n$, $z = \varphi_n^h$, $w = \psi_n^h$. Summing the two equations and using identities
\[
  \begin{split}    
[D^2u]_n &= \frac1{\tstep}\left(\frac1{\tstep}(u_{n+1}-u_n)-\frac1{\tstep}(u_n-u_{n-1})\right)\\ [Du]_n &= \frac12\left(\frac1{\tstep}(u_{n+1}-u_n)+\frac1{\tstep}(u_n-u_{n-1})\right)
  \end{split}
\]
we have that
\[
\frac1{\tstep}(E_{n+1}-E_n) +\inner{\begin{pmatrix}
  \varphi^h_n \\\psi^h_n
\end{pmatrix}}{ B(\partial^\tstep_t) \begin{pmatrix}
  \varphi^h \\\psi^h
\end{pmatrix}(t_n)} = \Inner{g_n}{[Du^h]_n}+
\inner{\begin{pmatrix}
  \varphi^h_n \\\psi^h_n
\end{pmatrix}}{
\begin{pmatrix}
  \rho_n\\ \sigma_n
\end{pmatrix}
}.
\]

Next, we multiply the $n$th equation by $\varrho^{2n}$ with $\varrho = e^{-\tstep/T} < 1$ and sum over $n$. Using that $\varrho^{2n}E_n-\varrho^{2(n-1)}E_n \leq 0$ we have that
\[
  \begin{split}
    \varrho^{2m}E_{m+1} \leq& E_1-\tstep \sum_{n = 1}^{m}\varrho^{2n}
\inner{\begin{pmatrix}
  \varphi^h_n \\\psi^h_n
\end{pmatrix}}{ B(\partial^\tstep_t) \begin{pmatrix}
  \varphi^h \\\psi^h
\end{pmatrix}(t_n)}\\
&+\tstep\sum_{n = 1}^{m} \varrho^{2n}\left(\Inner{g_n}{[Du^h]_n}+
\inner{\begin{pmatrix}
  \varphi^h_n \\\psi^h_n
\end{pmatrix}}{
\begin{pmatrix}
  \rho_n\\ \sigma_n
\end{pmatrix}
}\right).
\end{split}
\]
To apply the results of Theorem~\ref{thm:herglotz} and Lemma~\ref{lem:parseval}, we use the auxiliary sequences
\[
\varphi_n^{h,m} := \left\{
  \begin{array}{cc}
    \varphi_n^h & n \leq m\\
    0& n > m
  \end{array}
\right.
\qquad
\psi_n^{h,m} := \left\{
  \begin{array}{cc}
    \psi_n^h & n \leq m\\
    0& n > m
  \end{array}
\right..
\]
From Theorem~\ref{thm:herglotz} we have that
\[
  \begin{split}
     \sum_{n = 0}^{m}\varrho^{2n}
\inner{\begin{pmatrix}
  \varphi^h_n \\\psi^h_n
\end{pmatrix}}{ B(\partial^\tstep_t) \begin{pmatrix}
  \varphi^h \\\psi^h
\end{pmatrix}(t_n)} =&     \sum_{n = 0}^{\infty}\varrho^{2n}
\inner{\begin{pmatrix}
  \varphi^{h,m}_n \\ \psi^{h,m}_n
\end{pmatrix}}{ B(\partial^\tstep_t) \begin{pmatrix}
  \varphi^{h,m} \\\psi^{h,m}
\end{pmatrix}(t_n)}\\
 \geq& \beta\tstep \sum_{n = 0}^\infty\varrho^{2n}
\left(\|(\partial_t^{\tstep})^{-1} \varphi^{h,m}(t_n)\|_{-1/2,\Gamma}^2\right. \\&+\left. \|(\partial_t^{\tstep})^{-1} \psi^{h,m}(t_n)\|_{1/2,\Gamma}^2\right).
  \end{split}
\]
Whereas from Lemma~\ref{lem:parseval} we have
\[
  \begin{split}    
\left|\sum_{n = 0}^{m} \varrho^{2n}
\inner{\begin{pmatrix}
  \varphi^h_n \\\psi^h_n
\end{pmatrix}}{
\begin{pmatrix}
  \rho_n\\ \sigma_n
\end{pmatrix}
}\right|&
=\left|\sum_{n = 0}^{\infty} \varrho^{2n}
\inner{\begin{pmatrix}
  \varphi^{h,m}_n \\\psi^{h,m}_n
\end{pmatrix}}{
\begin{pmatrix}
  \rho_n\\ \sigma_n
\end{pmatrix}
}\right|\\
&\hspace{-2cm}\leq \frac{\beta \tstep}{2}\sum_{n = 0}^{\infty} \varrho^{2n}\left(
\|(\partial_t^{\tstep})^{-1}\varphi^{h,m}(t_n)\|_{-1/2,\Gamma}^2+
\|(\partial_t^{\tstep})^{-1}\psi^{h,m}(t_n)\|_{1/2,\Gamma}^2
\right)\\
&+ \frac{1}{2\beta}\tstep\sum_{n = 0}^{\infty} \varrho^{2n}\left(
\|\partial_t^{\tstep}\rho(t_n)\|_{1/2,\Gamma}^2+
\|\partial_t^{\tstep}\sigma(t_n)\|_{-1/2,\Gamma}^2
\right).
  \end{split}
\]

Further
\[
  \begin{split}
    \tstep\sum_{n = 1}^m \varrho^{2n}\Inner{g_n}{[Du^h]_n} &\leq
    \tstep\sum_{n = 1}^m \varrho^{2n}\left(\frac12c_1^2\|g_n\|^2+ \frac12  (E_{n+1}+E_n)\right)\\
&\leq
    \tstep\frac12c_1^2\sum_{n = 1}^m \varrho^{2n}\|g_n\|^2+ \tstep\sum_{n = 1}^m \varrho^{2n}E_n
+\frac12\tstep \varrho^{2m}E_{m+1}.
  \end{split}
\]

Combining everything we have that 
\[
  \begin{split}    
  \frac{1-\tstep}{\varrho^2}\varrho^{2(m+1)} E_{m+1}+&\frac{\beta}{2}\tstep \sum_{n = 0}^{m}\varrho^{2n}
\left(\|(\partial_t^{\tstep})^{-1} \varphi^h(t_n)\|_{-1/2,\Gamma}^2+ \|(\partial_t^{\tstep})^{-1} \psi^h(t_n)\|_{1/2,\Gamma}^2\right)\\ &\leq E_1+R_N+\tstep \sum_{n = 1}^m \varrho^{2n}E_n
  \end{split}
\]
where
\[
    R_N =  \frac{1}{2\beta}\tstep\sum_{n = 0}^{\infty} \varrho^{2n}\left(
\|\partial_t^{\tstep}\rho(t_n)\|_{1/2,\Gamma}^2+
\|\partial_t^{\tstep}\sigma(t_n)\|_{-1/2,\Gamma}^2
\right)
+ \tstep\frac12c_1^2\sum_{n = 1}^N \varrho^{2n}\|g_n\|^2.
\]
Recalling $\varrho^{2N} = e^{-2}$ and $\tstep \leq \tstep_0 < 1$, Gronwall lemma finishes the proof.
\end{proof}

Convergence now follows from stability and consistency. To state the theorem we let 
\begin{equation}
  \label{eq:errs} 
e^{h,u}_n :=  u^h_n-R_hu(t_n), \quad e^{h,\varphi} :=  \varphi^h_n-P^2_h\varphi(t_n),
\quad e^{h,\psi} :=  \psi^h_n-P^3_h\psi(t_n), 
\end{equation}
where $R_h$ is the elliptic projection \eqref{eq:Rh}, and 
\[
  \begin{split}    
P_h^2\colon  H^{-1/2}(\Gamma) \to X_h^{-\frac12} \quad
P_h^3\colon  H^{1/2}(\Gamma) \to X_h^{\frac12} 
  \end{split}
\]
are $L^2(\Gamma)$ projections.

\begin{theorem}\label{thm:conv}
Let $u$,$\varphi$ and $\psi$ be the solution of the continuous problem \eqref{eq:prob} and  $u_n^h$, $\varphi_n^h$, $\psi_n^h$, $n = 0,\dots,N$, the solution of the fully discrete system \eqref{eq:full_disc}. If $u \in C^4([0,T]; L^2(\Omega)) \cap C^3([0,T]; H^1(\Omega))$,  $\psi \in W_0^{m}([0,T]; H^{1/2}(\Gamma))$ and $\varphi \in W_0^m([0,T]; H^{-1/2}(\Gamma))$, then under the CFL condition \eqref{eq:CFL} and with \eqref{eq:errs} 
\[
\left\|\frac{e^{h,u}_n-e^{h,u}_{n-1}}{\tstep}\right\|_\Omega+\left\|\frac{ \nabla e^{h,u}_n+ \nabla e^{h,u}_{n-1}}{2}\right\|_\Omega  
\leq C(T) \mathcal{E}_n + \O(\tstep^2)
\]
and
\[
\tstep \sum_{n = 0}^{N}
\left(\|(\partial_t^{\tstep})^{-1} e^{h,\varphi}(t_n)\|_{-1/2,\Gamma}^2+ \|(\partial_t^{\tstep})^{-1} e^{h,\psi}(t_n)\|_{1/2,\Gamma}^2\right)\leq C(T) \mathcal{E}_n + \O(\tstep^2)
\]
where
\[
\begin{split}
\mathcal{E}_n &:= 
 \left(\tstep \sum_{n = 0}^{N} \|(I-R_h)\partial_t^2u(t_n)\|^2_\Omega+ \|(I-R_h)u(t_n)\|^2_\Omega\right)^{1/2}\\
&+ \max_{t \in [0,t_n]}\|(P_h^2-I)\partial_t^{m_0}\partial_t^{\tstep}\varphi(t)\|_{-1/2}+ \|(P_h^3-I)\partial_t^{m_0}\partial_t^{\tstep}\psi(t)\|_{1/2},
  \end{split}
\]
with $m_0 > 5$ and $m > 9$ for the trapezoidal rule and $m_0> 3$ and $m > 7$ for the other CQ methods satisfying Assumption~\ref{ass:delta}.
\end{theorem}
\begin{proof}
 The errors $e^{h,u}$, $e^{h,\varphi}$ and $e^{h,\psi}$ satisfy the discrete system from Theorem~\ref{thm:stab} with perturbations given by
\[
  \begin{split}    
g_n =& c^{-2}[D^2 R_hu]_n-c^{-2}\partial_t^2u(t_n)+R_hu(t_n)-u(t_n)\\
\begin{pmatrix}
  \rho_n\\ \sigma_n
\end{pmatrix}
=&
\begin{pmatrix}
[D \gamma^-u]_n-\partial_t \gamma^-u(t_n)\\
0  
\end{pmatrix}
+B(\partial_t) \begin{pmatrix}
  P_h^2\varphi-\varphi \\P_h^3\psi-\psi
\end{pmatrix}(t_n)\\
&+(B(\partial_t^\tstep)-B(\partial_t)) \begin{pmatrix}
  P_h^2\varphi \\P_h^3\psi
\end{pmatrix}(t_n).
  \end{split}
\]
For  $u \in C^4([0,T]; L^2(\Omega))$ 
\[
\|g_n\| \leq C\left(\|(I-R_h)\partial_t^2u(t_n)\|_\Omega+\|(I-R_h)u(t_n)\|_\Omega\right)+\O(\tstep^2).
\]
Further, for $u \in C^3([0,T]; H^1(\Omega))$ 
\[
  \begin{split}    
\|[D \gamma^-u]_n-\partial_t \gamma^-u(t_n)\|_{H^{1/2}(\Gamma)} &\leq C
\|[D u]_n-\partial_t u(t_n)\|_{H^1(\Omega)} = \O(\tstep^2).
  \end{split}
\]

Assuming $\psi \in W^{m_0}_0([0,T]; H^{1/2}(\Gamma))$ and $\varphi \in W_0^{m_0}([0,T]; H^{-1/2}(\Gamma))$ we have from \eqref{eq:Bcont} and \cite[Lemma~2.5]{lbbook} (see also \cite[Lemma~2.2]{Lub94} and \cite{FJStrap}), 
\[
  \begin{split}    
\left\|\partial_t^{\tstep}B(\partial_t) \begin{pmatrix}
  P_h^2\varphi-\varphi \\P_h^3\psi-\psi
\end{pmatrix}(t_n)\right\|
\leq& C \int_0^{t_n} \|(P_h^2-I)\partial_t^{m_0}\partial_t^{\tstep}\varphi(t)\|_{-1/2}dt\\&+ C \int_0^{t_n}\|(P_h^3-I)\partial_t^4\partial_t^{\tstep}\psi(t)\|_{1/2}dt\\
\leq& C t_n \max_{t \in [0,t_n]}\|(P_h^2-I)\partial_t^{m_0}\partial_t^{\tstep}\varphi(t)\|_{-1/2}\\&+Ct_n\max_{t \in [0,t_n]} \|(P_h^3-I)\partial_t^{m_0}\partial_t^{\tstep}\psi(t)\|_{1/2},
  \end{split}
\]
where $m_0 > 5$ for the trapezoidal rule and $m_0> 3$ for other CQ methods.
Further from \cite[Theorem~2.2 and 2.3]{lbbook} (see also \cite[Theorem~3.1]{Lub94} and \cite[Theorem 2.1]{FJStrap}), for $\psi \in W^{m}_0([0,T]; H^{1/2}(\Gamma))$ and $\varphi \in W_0^{m}([0,T]; H^{-1/2}(\Gamma))$
\[
  \begin{split}    
\left\|\partial_t^{\tstep}(B(\partial_t^\tstep)-B(\partial_t)) \begin{pmatrix}
  P_h^2\varphi \\P_h^3\psi\end{pmatrix}(t_n)\right\| \leq& \left\|\partial_t^{\tstep}B(\partial_t^\tstep)-\partial_tB(\partial_t) \begin{pmatrix}
  P_h^2\varphi \\P_h^3\psi\end{pmatrix}(t_n)\right\|\\
&+\left\|(\partial_t-\partial_t^{\tstep})B(\partial_t) \begin{pmatrix}
  P_h^2\varphi \\P_h^3\psi\end{pmatrix}(t_n)\right\|
\\=& \O(\tstep^2),
  \end{split}
\]
where $m > 9$ for the trapezoidal rule and $m > 7$ for the other CQ schemes.

Combining all the estimates with stability from Theorem~\ref{thm:stab} 
it just remains to bound the initial error:
\[
E_0(e^{h,u}) = 
 \frac12 \left \|\frac{e^{h,u}_1-e^{h,u}_0}{\tstep}\right\|^2_{\Omega}
+ \frac12 \Inner{\kappa \nabla e^{h,u}_1}{\nabla e^{h,u}_0}.
\]
Using the Taylor expansion in time of the exact solution and the choice of $u_0^h$ and $u_1^h$ shows that $E_0 = \O(\tstep^2)$ and thus completes the proof.
\end{proof}


\section{Implementation and choice of linear multistep method}

Let $\{v_1,\dots,v_{M_1}\}$,   $\{z_1,\dots,z_{M_2}\}$, and $\{w_1,\dots,w_{M_3}\}$ be the bases of $X_h$, $X_h^{-\frac12}$ and  $X_h^{\frac12}$ respectively. Denote the mass and trace matrices by
\[
(\Mb)_{ij} = \Inner{\frac1{c^2}v_i}{v_j}, \quad (\Cb)_{ij} = \inner{z_i}{\gamma^- v_j}, \quad
(\Ib)_{ij} = \inner{z_i}{w_j}.
\]
At each time step the following system needs to be solved
\[
  \begin{split}
    \frac1{\tstep^2} \Mb \ub_{n+1} + \Cb^T \bm{\varphi}_n &= \text{known terms}\\ 
    \begin{pmatrix}
      -\frac1{2\tstep} \Cb \ub_{n+1} +\frac12 \Ib \bm{\psi}_n\\
-\frac12 \Ib^T \bm{\varphi}_n
    \end{pmatrix}
+ \mathbf{B}_h \left(\frac{\delta(0)}{\tstep}\right)
\begin{pmatrix}
  \bm{\varphi}_n\\
  \bm{\psi}_n\\
\end{pmatrix}
&= \text{known terms}
  \end{split}
\]
We can either solve a large system for both domain and boundary unknowns in each timestep, alternatively eliminating $\ub_{n+1}$ we obtain the system for the densities at time $t_n$
\[
\left[  \begin{pmatrix}
    \frac{\tstep}{2} \Cb \Mb^{-1}\Cb^T & \frac12 \Ib\\
-\frac12 \Ib^T & 0 
  \end{pmatrix}
+ \mathbf{B}_h \left(\frac{\delta(0)}{\tstep}\right)\right]
\begin{pmatrix}
  \bm{\varphi}_n\\
  \bm{\psi}_n\\
\end{pmatrix}
= \text{known terms}.
\]
Solving the linear system iteratively, the matrix $\Cb \Mb^{-1}\Cb^T$ need not be constructed explicitly. If using mass lumping for the  conforming finite element method or using a symmetric discontinuous Galerkin method  to discretize the interior equations, the mass matrix $\Mb$ becomes block diagonal and the product of the matrices could be constructed efficiently.  

The main cost is however in computing the history contained in the known terms
\[
\sum_{j = 0}^{n-1} \bm{\omega}_{n-j} \begin{pmatrix}
  \bm{\varphi}_j\\
  \bm{\psi}_j\\
\end{pmatrix},
\]
where the weights can be approximated by (see \eqref{eq:omega_quad})
\begin{equation}
  \label{eq:Omega_quad}  
 \bm{\omega}_j(K) 
 \approx \frac{\kr^{-j}}{N+1} \sum_{\ell = 0}^N 
\mathbf{B}_h\left(\frac{\delta(\kr\zeta_{N+1}^{-\ell})}{\Delta  t}\right)
\zeta_{N+1}^{\ell j}.
\end{equation}
The simplest implementation involves computing all the weights using FFT (cost $\O(N \log N)$) and then at each step evaluating the whole history. This results in a cost of $\O(N^2)$ and heavy use of computer memory. A more efficient approach is given in  \cite{HaLuSc:1985}, which reduces the cost to $\O(N \log^2 N)$. With additional modifications described in \cite{Banjai:2010}, the weights need never be computed and data sparse methods such as $\mathcal{H}$-matrices and the fast multipole method can be used to accelerate the discretization in space; see \cite{BaKa:2014a}. 

An aspect of the implementation of CQ for boundary integral equations that is rarely mentioned, is the spatial quadrature required to compute the integral operators. The reason for this is that unlike \bl{in space-time Galerkin methods, where quadrature needs to carefully deal with the sharp space-time cone \cite{Shanz_quad:2021}}, the spatial quadrature in CQ schemes is usually straightforward. \bl{The reason for this is that the CQ smooths out the space-time cone and the quadrature required is the same as needed for steady state problems, where quadrature techniques are well-developed; see  \cite{SaSc:2011}}. This is however not the case if the frequencies
\[
s_\ell = \frac{\delta(\kr\zeta_{N+1}^{-\ell})}{\Delta  t}
\]
become large compared to the spatial meshwidth. Precisely this  can occur in the case of the trapezoidal CQ and is exacerbated by the CFL requirement \eqref{eq:CFL}. We explain in more detail next.

The error in \eqref{eq:Omega_quad} is of size $\O(\kr^{N+1})$, however due to finite-precision considerations the parameter $\kr$ cannot be chosen too small. \bl{ Indeed, if $\eps$ denotes the finite precision, due to the multiplication by $\lambda^{-j}$, the total error is of the form
\[
\kr^{N+1}+\kr^{-N}\eps.
\]
Optimizing the choice of $\kr$, see \cite{Lubich:1988b}, gives $\kr = \eps^{\frac1{2N+1}}$. } \bl{Hence, $\kr < 1$, and since $N^{-1} = \O(\tstep)$} we have  $|\kr-1| = \O(\tstep)$.  For $\delta(\zeta)$ analytic in a neighbourhood of $|\zeta| \leq 1$, \bl{as is the case for BDF2}, we have that $s_\ell = \O(\tstep^{-1})$. On the contrary, for the trapezoidal rule we have \bl{due to the existence of a pole at $\zeta = 1$}, that
\[
|\dtr(\kr\zeta)| \leq \frac{4}{1-\kr} = \O(\tstep^{-1})
\]
and hence $s_\ell = \O(\tstep^{-2})$; see \cite{lbbook,FJStrap}.  Recalling that the kernel of  the integral operators in the frequency domain is given by $\frac{e^{-s_\ell r}}{4\pi r}$ in 3D, we see that spatial quadrature is needed for kernels that are either highly oscillatory and/or strongly decaying. 
The support of the boundary element functions is contained on panels of size $h \propto \tstep$, hence either a special spatial quadrature is needed or increasing number of nodes as the time-step is decreased. 

The above discussion strongly suggests the use of BDF2 instead of the trapezoidal scheme. However, the trapezoidal scheme is conservative, whereas BDF2 is strongly dissipative. This can be seen by examining the expansion of $\delta(e^{\mi \omega \tstep})$:
\[
\begin{split}  
\dbdf(e^{-\mi \omega \tstep}) &= \mi \omega \tstep+\mi \frac13\omega^3\tstep^3+\frac14 \omega^4\tstep^4 + \O((\omega\tstep)^5)\\
\dtr(e^{-\mi \omega}) &= \mi \omega\tstep +\mi \frac1{12}\omega^3\tstep^3+\O((\omega \tstep)^5).
\end{split}
\]
Namely, the action of CQ is to replace the exact wave number $\mi \omega$ by the approximation $\frac{\delta(e^{-\mi \omega\tstep})}{\tstep}$; \bl{see \eqref{eq:sdt}}. \bl{Thus, this action replaces  the exact kernel (in 3D) 
$
\frac{e^{-\mi \omega r}}{4\pi r} \text{ by }  \frac{e^{- \frac{\delta(e^{-\mi \omega \tstep})}{\tstep}r}}{4\pi r},
$
}
where it can be seen that BDF2 with $\Re \dbdf(e^{-\mi \omega\tstep}) = \frac14 \omega^4\tstep^4 + \O((\omega\tstep)^6) >0$ introduces damping and is dissipative, whereas the fact that the trapezoidal rule is conservative can be seen from  the identity $\Re \dtr(e^{-\mi \omega\tstep}) \equiv 0$. We can also see in the above expansions that both schemes are of second order, but that the error constant, \bl{i.e., the constant in front of $(\omega \tstep)^3$},  for the BDF2 scheme is $\frac13$ and for the trapezoidal rule is the optimal $\frac1{12}$. 

We show next how to construct a scheme that retains to a high degree the positive properties of the above two schemes. Namely, we search for a truncated trapezoidal rule of the form
\[
\dttr(\zeta) = (1-\zeta)+\frac12 (1-\zeta)^2+\sum_{j = 2}^{J} 2^{-j}c_j (1-\zeta)^{j+1}
\]
for some constants $c_j$ to be determined. Note that setting $c_j = 0$, we obtain again the BDF2 scheme, whereas setting $c_j = 1$ gives the truncated expansion of the trapezoidal generating function. For any choice of $c_j$, the scheme is second order \bl{and as $\dttr(\zeta)$ is entire,  the frequencies $s_\ell = \O(\tstep^{-1})$ grow only linearly as $\tstep \to 0$. } However, A-stability is not guaranteed. 

\bl{We make use of an interior point algorithm for constrained optimization as implemented in Matlab's \texttt{fmincon} to minimize the error constant in the resulting method under the condition of A-stability and with $c_j$ restricted to $[0,1]$. A-stability is checked numerically by sampling $\Re \dttr(e^{-\mi x})$ for $5\times 10^{4}$ equally spaced points in the interval $x \in [0,\pi]$.} This results for $J = 4$ in coefficients
\[
c_2 = 0.893817850529318,\;   c_3 = 0.684154908023834, \; c_4 =    0.629642997466429
\]
\bl{with the above described numerical test of A-stability indicating that $\Re \dttr(e^{-\mi x}) \geq -5 \times 10^{-17}$.}
With this choice of coefficients, the expansion of $\dttr(e^{-\mi \omega \tstep})$ is
\[
\dttr(e^{-\mi \omega \tstep}) = \mi \omega\tstep +\mi \frac1{9.10\cdots}\omega^3\tstep^3+(3.37\cdots)\times 10^{-4}\omega^4\tstep^4 + \O((\omega\tstep)^5).
\]
As expected the error constant \bl{in front of $(\omega \tstep)^3$} is between those of BDF2 ($\frac13$) and the trapezoidal rule ($\frac1{12}$) and the scheme is dissipative but much less so  than BDF2 as 
\[
\Re \dttr(e^{-\mi \omega \tstep}) = (3.37\cdots)\times 10^{-4}\omega^4\tstep^4+\O((\omega\tstep)^6)
\]
\bl{
compared to
\[
\Re \dbdf(e^{-\mi \omega\tstep}) = \frac14 \omega^4\tstep^4 + \O((\omega\tstep)^6).
\]
}
The stability regions of the three methods are shown in Figure~\ref{fig:stab_reg}.

\bl{Other choices of $J$ could be used. Increasing $J$ would lead to a better error constant, but as indicated in the plot above, would likely lead to a growing boundary of the stability region and larger frequencies $s_\ell$. Thus, a choice of $J$ is a compromise and the optimal choice will also depend on the particular wave problem that it is applied to.}

\begin{figure}
  \centering
  \includegraphics[width=.87\textwidth]{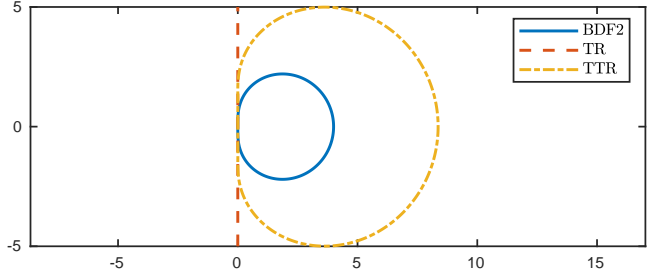}
  \caption{The boundaries of the stability regions of the BDF2, trapezoidal and numerically obtained truncated trapezoidal schemes. Stability region of the trapezoidal rule is the left-half complex plane, whereas for the BDF2 and truncated trapezoidal it is the region outside of the bounded domains shown.}
  \label{fig:stab_reg}
\end{figure}

\section{Numerical experiments}
To test the theoretical results, we first let $\Omega$ be the disk centred at the origin with radius $3$. We set $f \equiv 0$, $c \equiv 1$, $\kappa \equiv
\begin{pmatrix}
  1 & 0\\0& 1
\end{pmatrix}
$, $\uinc \equiv 0$ and initial data
\[
u_0(x) = \exp(-2|x|^2), \qquad v_0 \equiv 0.
\]
Note that the support of $u_0$ is not contained in $\Omega$, but is approximately zero on $\partial\Omega$, so that this discrepancy has no significant effect on the numerical experiments. Using the Hankel transform, we find that the exact solution is given by
\[
u(x,t) = \frac14\int_0^\infty \exp(-\frac{k^2}{8})J_0(k|x|) k \cos(kt)dk, \qquad x \in \Omega^-,
\]
where $J_0(\cdot)$ is a Bessel function of the first kind. 

The smooth domain $\Omega$ is approximated by a polygonal boundary. For the spatial discretization, we set $X_h$ to be the space of piecewise-linear nodal finite element functions, $X_h^{-\frac12}$ the space of piecewise constant and $X_h^{\frac12}$ the space of piecewise linear boundary element functions. 
To ensure that the CFL condition \eqref{eq:CFL} is satisfied we estimate the largest eigenvalue $\lambda_{\text{max}} = (\Cinv/h)^2$ of the generalized eigenvalue problem
\[
\Inner{\kappa \nabla u}{\nabla v} = \lambda \Inner{u}{v}, \qquad \forall v \in X_h
\]
and set the time-step to $\tstep = 2/\sqrt{\lambda_{\text{max}}}$.

We denote by $I^h \colon C(\Omega) \to X_h$ be the nodal interpolant onto $X_h$. As the error measure we use the energy error
\[
  \begin{split}    
\text{error} =& \max_n \left\|\frac{u^h_n-u^h_{n-1}}{\tstep}-\frac{I^hu(t_n)-I^hu(t_{n-1})}{\tstep}\right\|_\Omega\\&+\left\|\nabla \left(\frac{u^h_n+u^h_{n-1}}{2}-\frac{I^hu(t_n)+I^hu(t_{n-1})}{2}\right)\right\|_\Omega.
  \end{split}
\]

The CQ time-discretization of the boundary integral operators is based on  BDF2, the trapezoidal rule or the truncated trapezoidal rule. Each of these methods, as well as leapfrog in the interior,  is of second order and with the choice of finite and boundary element spaces  we expect the error in Theorem~\ref{thm:conv} to be 
\[
\text{error} = \O(\tstep^2+h^{3/2}) = \O(h^{3/2}) = \O(\tstep^{3/2}).
\]
Here we used that $\tstep \propto h$ and the approximation properties of the piecewise linear  finite element space
\[
\inf_{v^h \in X_h}\|v-v^h\|_\Omega \leq C h^2 \|v\|_{H^2(\Omega)}
\]
for any $v \in H^2(\Omega)$, and of the piecewise constant and piecewise linear  boundary element spaces:
\[
\inf_{\varphi^h \in X_h^{-\frac12}}\|\varphi-\varphi^h\|_{-1/2,\Gamma} \leq C h^{3/2} \|\varphi\|_{1,\Gamma}
\quad
\inf_{\psi^h \in X_h^{\frac12}}\|\psi-\psi^h\|_{1/2,\Gamma} \leq C h^{3/2} \|\psi\|_{2,\Gamma};
\]
for any $\varphi \in H^1(\Gamma)$ and $\psi \in H^2(\Gamma)$; see \cite{BrenSc} and \cite{SaSc:2011} respectively.
The polygonal approximation of the boundary does not destroy this convergence  rate  indeed it adds an $\O(h^2)$ additional error; see \cite{BrenSc,SaSc:2011}. 

\begin{figure}
  \centering
  \includegraphics[width=0.8\textwidth]{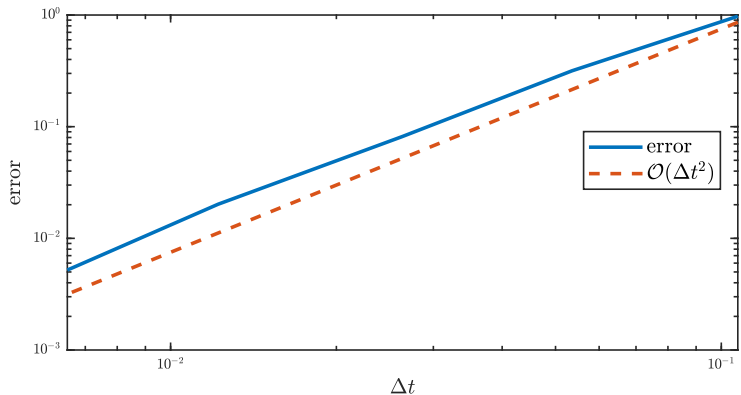}
  \caption{Convergence of the error agains the time-step for the unit disk.}
  \label{fig:conv_disk}
\end{figure}

The convergence of the error is plotted in Figure~\ref{fig:conv_disk}. Only a single graph is shown since all the CQ schemes give the same error. However, we have used significantly more quadrature points per element when computing the boundary element matrices for the trapezoidal scheme than for the other two schemes. The fact that the error is the same for the three CQ schemes in this example, is not surprising as quite small time-steps are used due to the CFL condition and as the solution is quite simple with no reflections of the wave; see the discussion in \cite{Banjai:2014}. A clear second order convergence can be seen. This is not a contradiction to the theory and is likely to be a super-convergence effect due to compairing the discrete solution to the interpolation of the exact solution.

We end with a numerical example concerning a non-convex domain, namely the L-shape domain with vertices $\{(-1,-1), (1,-1), (1,3), (-3,3), (-3,1), (-1, 1)\}$. We take a piecewise constant wave speed in $\Omega$:
\[
c(x) = \left\{
  \begin{array}{cc}
    2 &  \qquad x_1 \in (-2.5,1) \text{ and } x_2 \in (1.5, 2.5)\\
    1 & \text{otherwise}. 
  \end{array}
\right.
\]
We again set $f \equiv 0$, $c \equiv 1$, $\kappa \equiv
\begin{pmatrix}
  1 & 0\\0& 1
\end{pmatrix}
$ and this time zero initial condition in $\Omega$, i.e., $u_0 = v_0 = 0$. 
The scattering problem is excited by 9 sources outside of the domain that focus at a point inside $\Omega$. Images of the solution in $\Omega$ are shown in Figure~\ref{fig:focus_plot}. \bl{No exact solution is available, but further numerical experiments not shown here suggest that the images shown in Figure~\ref{fig:focus_plot} are accurate.}

\begin{figure}
  \centering
  \includegraphics[width=0.9\textwidth]{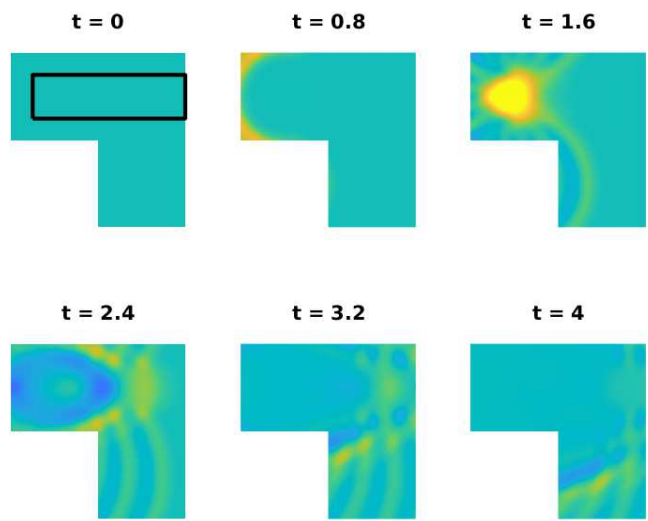}
  \caption{Waves focusing at a point in $\Omega$. In the section of the L-shaped domain indicated by a rectangle in the first image the wave speed is twice as fast as in the \bl{remainder} of the domain.}
  \label{fig:focus_plot}
\end{figure}
\def\cprime{$'$}

\end{document}